\documentclass[11pt,a4paper,twoside]{article}

\usepackage[UKenglish]{babel}

\usepackage[T1]{fontenc}

\usepackage[utf8]{inputenc}

\usepackage{latexsym,amsfonts,amsmath,amsthm,amssymb,mathrsfs}

\usepackage{fullpage}

\usepackage{xcolor,bm,bbm}

\usepackage{paralist}

\usepackage{todonotes}

\usepackage{fourier}

\usepackage{mathtools}

\mathtoolsset{showonlyrefs}

\newtheorem{thm}{Theorem}[section]

\newtheorem{cor}[thm]{Corollary}

\newtheorem{proposition}[thm]{Proposition}

\newtheorem{rem}[thm]{Remark}

\newtheorem{example}[thm]{Example}

\newtheorem{lemma}[thm]{Lemma}

\newtheorem{conjecture}[thm]{Conjecture}

\newtheorem{thmalpha}{Theorem}

\allowdisplaybreaks

\begin{document}

\title{Cyclic group tensors and the volume ratio\\ of projective tensor products}

\author{Joscha Prochno\footnote{University of Passau, Faculty of Computer Science and Mathematics, Dr.-Hans-Kapfinger-Stra{\ss}e 30, 94032 Passau, Germany. Email: joscha.prochno@uni-passau.de}}

\date{}

\maketitle

\begin{abstract}
We revisit the problem of estimating the norm of the natural embedding from the Euclidean space $\ell_2^{n^3}$ into the projective tensor product $\ell_p^n\otimes_\pi\ell_q^n\otimes_\pi\ell_r^n$ in the range $1\le p\le q\le 2\le r\le\infty$, studied by Giladi, Prochno, Sch\"utt, Tomczak-Jaegermann and Werner [J. Funct. Anal. 273 (2017), 471--495]. Using convolution on finite cyclic groups, we establish new lower bounds that improve their estimates and determine the exact asymptotic order of the embedding norm whenever $1/p+1/q+1/r\ge 2$. We then apply an estimate of Boedihardjo for the injective norm of random tensors to a randomly signed cyclic convolution tensor and use multilinear interpolation to determine the optimal exponent throughout the entire parameter range, up to logarithmic factors. As consequences, we obtain corresponding improvements and asymptotically sharp estimates, up to logarithmic factors, for the volume ratio of these projective tensor products.

\bigskip

\ 


\end{abstract}

\section{Introduction \& Main Results}\label{sec:introduction}

In the geometry of Banach spaces, the volume ratio $\operatorname{vr}(X)$ of a finite-dimensional normed space $X$ compares the volume of its unit ball with that of its John ellipsoid, the unique ellipsoid of maximal volume contained in the unit ball. It is a fundamental concept in the local theory of Banach spaces and asymptotic geometric analysis with important applications in approximation theory (see, e.g., \cite{PajorTomczak1989}). The notion goes back to the seminal works of Szarek \cite{Szarek1978} and Szarek and Tomczak-Jaegermann \cite{SzarekTomczak1980}, which were inspired by Kashin's celebrated result on nearly Euclidean decompositions of $\ell_1^{2n}$ \cite{kashin1977}; nearly Euclidean is to be understood in the sense of Banach--Mazur distance. To be precise, an important connection between what is now known as the volume ratio and the existence of nearly Euclidean decompositions was already implicit in the work of Szarek \cite{Szarek1978}, namely that every $2n$-dimensional Banach space $X$ admits a decomposition into two $n$-dimensional subspaces whose Banach--Mazur distances to $\ell_2^n$ are bounded by a universal constant times $\operatorname{vr}(X)^2$. Building on this observation, Szarek and Tomczak-Jaegermann \cite{SzarekTomczak1980} formally introduced the notion of volume ratio and investigated its relationship with the cotype $2$ constant of Banach spaces. They also established Kashin-type decompositions for several classes of Banach spaces, including the projective tensor products $\ell_p^n\otimes_\pi\ell_2^n$ for $1\le p\le2$. 

The foundations of the modern theory of tensor products of Banach spaces were laid by Grothendieck in his seminal 1953 Résumé \cite{Grothendieck1953} and had a huge impact on Banach space theory. This impact and the success of the concept of tensor products is to a large extent due to the work \cite{LindenstraussPelczynksi1968} of J. Lindenstrauss and A. Pe{\l}czy\'nski in the late sixties who reformulated Grothendieck’s ideas in the context of operator ideals and made this theory accessible to a broader audience. Modern applications include quantum computing and quantum information theory \cite{AubrunSzarek2006,AubrunSzarekWerner2010,AubrunSzarek2011,SzarekWernerZyczkowski2010} or the theory of locally decodable codes \cite{Efremenko2009}.

\subsection{Main Results \& Previous Work}

In this paper, we study the volume ratio of three-fold projective tensor products of finite-dimensional $\ell_p$ spaces. Formally, the volume ratio is defined for an $n$-dimensional normed space $X$ by
\[
 \operatorname{vr}(X)
 :=\left(\frac{\operatorname{vol}_n(B_X)}
 {\operatorname{vol}_n(\mathcal E_X)}\right)^{1/n},
\]
where $\mathcal E_X$ is the ellipsoid of maximal volume contained in the unit ball $B_X$ of $X$. In the case of projective tensor products of finite-dimensional  sequence spaces, the determination of this parameter can be reduced to norm estimates for a natural Euclidean embedding. More precisely, Giladi, Prochno, Sch\"utt, Tomczak-Jaegermann and Werner \cite{GPSW2017} studied the volume ratio of the three-fold projective tensor products
\[
 X_{p,q,r}^n:=\ell_p^n\otimes_\pi\ell_q^n\otimes_\pi\ell_r^n, 
 \qquad 1\le p\le q\le r\le\infty.
\]
They obtained sharp estimates in several parameter regimes and applications to nearly Euclidean decompositions and cotype.  In particular, they established asymptotically sharp estimates for the volume radii of the unit balls of $X_{p,q,r}^n$ throughout the range $1\le p\le q\le r\le\infty$; see \cite[Theorem~B]{GPSW2017}. Consequently, determining the volume ratio was reduced to identifying the John ellipsoids. Since the spaces $X_{p,q,r}^n$ have enough symmetries (see \cite{Schuett1982, SzarekTomczak1980}), their John ellipsoids are Euclidean balls of the form
\[
\mathcal E_{X_{p,q,r}^n}
=
\Bigl\|\operatorname{id}:\ell_2^{n^3}\to X_{p,q,r}^n\Bigr\|^{-1} 
B_2^{n^3}.
\]
Thus, the remaining task was to estimate the operator norm of the identity mapping $\operatorname{id}:\ell_2^{n^3}\to X_{p,q,r}^n$; see \cite[Theorem~C and Corollary~5.7]{GPSW2017}.  Their results leave open the precise order of the volume ratio in most parts of the range $1\le p\le q\le2\le r\le\infty$; some boundary cases are covered in \cite[Remark~2.1 and Corollary~5.7]{GPSW2017}. 

The purpose of the present paper is to sharpen the existing estimates in the range $1\le p\le q\le 2\le r\le\infty$ and, in particular, to determine the optimal asymptotic order throughout this parameter range, up to logarithmic factors. We identify the underlying vector space of $X_{p,q,r}^n$ with $\mathbb R^{n^3}$ through the canonical tensor basis $(e_i\otimes e_j\otimes e_k)_{1\le i,j,k\le n}$ and consider the identity operator
\[
 \operatorname{id}:\ell_2^{n^3}\longrightarrow X_{p,q,r}^n.
\]
In what follows, constants implicit in $\lesssim$ and $\asymp$ are independent of $n$ and may depend on the fixed parameters $p,q,r$. Whenever such a dependence is present, we indicate it by subscripts. 
 
We start by recalling the relevant estimates and results from \cite{GPSW2017}. By Lemmas~5.5 and~5.6 of \cite{GPSW2017},
\begin{equation}\label{eq:intro-old-upper}
 \left\|\operatorname{id}:\ell_2^{n^3}\longrightarrow X_{p,q,r}^n\right\|
 \lesssim_{p,q,r} n^{\min\{1/q+1/r,1\}}
\end{equation}
and
\begin{equation}\label{eq:intro-old-lower}
 \left\|\operatorname{id}:\ell_2^{n^3}\longrightarrow X_{p,q,r}^n\right\|
 \gtrsim_{p,q,r}
 \max\left\{n^{1/2+1/r},\,
 n^{\min\{1/q+1/r,1\}+1/p-1},\,n^{1/q}\right\}.
\end{equation}

In particular, the bounds were already sharp when $p=1$ or $q=2$, as well as in the boundary cases $r=2$ and $r=\infty$, as follows from Lemmas~5.5 and~5.6 of \cite{GPSW2017}. Our aim is to improve the lower bounds in the remaining parameter range.

Our first result is based on an explicit (deterministic) construction involving a tensor supported on the relation $j=i+k$ in the finite cyclic group $\mathbb Z_n$. The associated trilinear form can be analyzed using Young's convolution inequality, leading to the first main result. 

\begin{thmalpha}\label{thm:intro-deterministic}
Let $1\le p\le q\le2\le r\le\infty$. Then, for every $n\in\mathbb N$,
\begin{equation}\label{eq:intro-cyclic-lower}
 \left\|\operatorname{id}:\ell_2^{n^3}\longrightarrow X_{p,q,r}^n\right\|
 \ge n^{\min\{1,\,1/p+1/q+1/r-1\}}.
\end{equation}
In particular, if $1/p+1/q+1/r\ge2$, then
\begin{equation}\label{eq:intro-exact}
 \left\|\operatorname{id}:\ell_2^{n^3}\longrightarrow X_{p,q,r}^n\right\|
 \asymp_{p,q,r}n.
\end{equation}
\end{thmalpha}

The estimate \eqref{eq:intro-cyclic-lower} is proved in Section~\ref{sec:cyclic}; the upper estimate needed to obtain \eqref{eq:intro-exact} is the previously known bound \eqref{eq:intro-old-upper}. Thus Theorem~\ref{thm:intro-deterministic} extends the region in which matching bounds are available, rather than replacing the upper estimate of \cite{GPSW2017}. The deterministic construction unfortunately does not by itself settle the entire parameter range. Our second main result does so with a logarithmic loss.

\begin{thmalpha}\label{thm:intro-universal}
For all $1\le p\le q\le2\le r\le\infty$,
\begin{equation}\label{eq:intro-global}
 \frac{n^{\min\{1/q+1/r,1\}}}{(\log(en))^{4(1-1/q)}}
 \lesssim
 \left\|\operatorname{id}:\ell_2^{n^3}\longrightarrow X_{p,q,r}^n\right\|
 \lesssim_{p,q,r} n^{\min\{1/q+1/r,1\}}. 
\end{equation}
\end{thmalpha}

The upper bound in Theorem~\ref{thm:intro-universal} is again due to \cite[Lemma~5.5]{GPSW2017}; the lower bound is the new assertion. For its proof, we place independent random signs on the cyclic support. The essential probabilistic input is the estimate of Boedihardjo \cite[Corollary~1.4]{Boedihardjo2024} for the injective norm of random tensors with independent centred bounded entries; we note that this bound is similar to the bound by Bandeira and van Handel for the expected spectral norm of a random matrix with independent entries \cite{BandeiraHandel2016}. In our setting, every one-dimensional coordinate fibre of the support contains exactly one nonzero coefficient. Boedihardjo's result therefore yields a bound of order $(\log(en))^2$ for the expected Hilbertian injective norm. We then combine a suitable choice of signs with elementary endpoint estimates and multilinear interpolation (see \cite[Section~4.4]{BerghLofstrom1976}). The random-tensor inequality itself is not new; its use for this cyclic construction and the resulting lower embedding estimate are the ingredients of the present argument. In the region $1/p+1/q+1/r\ge2$, Theorem~\ref{thm:intro-deterministic} removes the logarithmic loss.

The corresponding consequences for volume ratio follow from the volume-radius estimate in \cite[Theorem~B]{GPSW2017} and the relation between volume ratio and the Euclidean embedding norm for spaces with enough symmetries, as used in \cite[(2.3)]{GPSW2017}.

\begin{cor}\label{cor:intro-volume}
For $1\le p\le q\le2\le r\le\infty$,
\begin{equation}\label{eq:intro-vr}
 \frac{n^{\min\{1/q-1/2,\,1/2-1/r\}}}
 {(\log(en))^{4(1-1/q)}}
 \lesssim \operatorname{vr}(X_{p,q,r}^n)
 \lesssim_{p,q,r}
 n^{\min\{1/q-1/2,\,1/2-1/r\}}.
\end{equation}
If, in addition, $1/p+1/q+1/r\ge2$, then
\[
 \operatorname{vr}(X_{p,q,r}^n)
 \asymp_{p,q,r}n^{1/2-1/r}.
\]
\end{cor}

Here the volume-radius asymptotics and the reduction to the embedding norm are inherited from \cite{GPSW2017}; the improvement comes from the new embedding lower bounds. 

\medskip

The paper is organized as follows. In Section~\ref{sec:background tensor products}, we recall some basic facts about tensor products of Banach spaces. Section~\ref{sec:cyclic} is devoted to the deterministic cyclic construction. In Section~\ref{sec:random}, we investigate its randomly signed counterpart and derive the logarithmic estimates via interpolation.

\section{Some background on tensor products of Banach spaces}\label{sec:background tensor products}

Let $(X,\|\cdot\|_X)$ and $(Y,\|\cdot\|_Y)$ be two Banach spaces. The algebraic tensor product $X\otimes Y$ can be realized as a space of linear functionals on the vector space of all bilinear forms on $X\times Y$. For $x\in X$, $y\in Y$, and a bilinear form $B$ on $X\times Y$, we define
\[
(x\otimes y)(B):=B(x,y).
\]
The projective tensor product $X\otimes_\pi Y$ is the algebraic tensor product $X\otimes Y$ equipped with the projective tensor norm
\[
\|A\|_{X\otimes_\pi Y}
:=
\inf\left\{
\sum_{i=1}^m \|x_i\|_X\|y_i\|_Y :
A=\sum_{i=1}^m x_i\otimes y_i,\;
x_i\in X,\;y_i\in Y
\right\}.
\]
Similarly, the injective tensor product $X\otimes_\varepsilon Y$ is the algebraic tensor product $X\otimes Y$ equipped with the injective tensor norm
\[
\|A\|_{X\otimes_\varepsilon Y}
:=
\sup\left\{
\left|\sum_{i=1}^m\varphi(x_i)\psi(y_i)\right| :
\varphi\in B_{X^*},\;
\psi\in B_{Y^*}
\right\},
\]
where $A=\sum_{i=1}^m x_i\otimes y_i$ is any representation of $A$; we may also use the short notation $\|\cdot\|_{\mathrm{inj}}$ if the underlying space is clear from the context. Throughout this paper, we restrict ourselves to finite-dimensional Banach spaces, so that the algebraic tensor products equipped with either of these norms are complete. In this setting, the projective tensor product $X\otimes_\pi Y$ can be identified isometrically with the space $\mathcal N(X^*,Y)$ of nuclear operators from $X^*$ into $Y$. Moreover, its dual can be identified isometrically with the space $\mathcal L(X,Y^*)$ of bounded linear operators from $X$ into $Y^*$, equipped with the operator norm.

The injective tensor product $X\otimes_\varepsilon Y$ can be identified isometrically with $\mathcal L(X^*,Y)$. In particular, for $A\in X\otimes_\varepsilon Y$, we have
\[
\|A\|_{X\otimes_\varepsilon Y}
=
\sup_{x^*\in B_{X^*}}\|Ax^*\|_Y. 
\]

We refer the reader to the wonderful books of \cite{DiestelFourieSwart2008} and \cite{Ryan2002} for more information about tensor products.

\section{A Lower Bound for the Identity Operator via Cyclic Convolution}\label{sec:cyclic}

Let $n\in\mathbb N$ and consider $X_{p,q,r}^n=\ell_p^n\otimes_\pi\ell_q^n\otimes_\pi\ell_r^n$ with $1\leq p\leq q\leq2\leq r\leq\infty$. We show that the upper bound for the norm of the identity obtained in \cite{GPSW2017} is sharp in the region
$\frac1p+\frac1q+\frac1r\geq2$.
The argument is based on convolution on the finite cyclic group $\mathbb Z_n:=\mathbb Z/n\mathbb Z$ of $n$ elements,
which we identify with $\{0,\ldots,n-1\}$, with addition understood modulo $n$. We use the additive structure of $\mathbb Z_n$ only to perform index arithmetic modulo $n$ and to define cyclic convolution, which allows us to apply Young's convolution inequality; no further group-theoretic properties are needed. Throughout, vectors indexed by $\mathbb Z_n$ are identified with vectors in $\mathbb R^n$, and all $\ell_s$-norms are taken with respect to counting measure, that is,
\[
    \|x\|_s := \left(\sum_{i\in\mathbb Z_n}|x_i|^s\right)^{1/s},
    \qquad 1\leq s<\infty,
\]
and
\[
    \|x\|_\infty:=\max_{i\in\mathbb Z_n}|x_i|.
\]
We recall Young's convolution inequality on the finite group $\mathbb Z_n$. If $1\leq a,b,c\leq\infty$ satisfy
\begin{equation}\label{eq:young-relation}
\frac1a+\frac1b=1+\frac1c,
\end{equation}
then
\begin{equation}\label{eq:young-convolution}
\|x*y\|_c\leq\|x\|_a\|y\|_b,
\end{equation}
where
\begin{equation}\label{eq:definition-cyclic-convolution}
(x*y)_k:=\sum_{i\in\mathbb Z_n}x_i y_{k-i}
\end{equation}
is the discrete convolution of $x$ and $y$. The constant in \eqref{eq:young-convolution} is one (see, e.g., \cite{Folland2016}).

Recall that, identifying the standard basis of $\mathbb R^n$
with $(e_i)_{i\in\mathbb Z_n}$, we may associate a trilinear
form with every tensor in $(\mathbb R^n)^{\otimes 3}$.
Indeed, using the standard Euclidean inner product on
$\mathbb R^n$, an elementary tensor $e_i\otimes e_j\otimes e_k$
corresponds to the trilinear form
\[
(x,y,z)\mapsto
\langle e_i,x\rangle
\langle e_j,y\rangle
\langle e_k,z\rangle
=x_i y_j z_k.
\]
By linearity, a general tensor
\[
A=\sum_{i,j,k\in\mathbb Z_n}
a_{ijk}\,e_i\otimes e_j\otimes e_k
\]
is therefore identified with the trilinear form
\[
T_A(x,y,z)
=\sum_{i,j,k\in\mathbb Z_n}
a_{ijk}x_i y_j z_k,
\qquad x,y,z\in\mathbb R^n.
\]
In what follows, we shall also use the notation $x_{+}:=\max\{x,0\}$, $x\in\mathbb R$. 

Our construction below is inspired by the classical structure tensor of a finite group algebra and, more specifically, by cyclic group tensors (see, e.g., Cohn and Umans \cite{CohnUmans2003}, Alexeev, Forbes, and Tsimerman \cite{AlexeevForbesTsimerman2011}, and Alman \cite{Alman2019}). For the cyclic group $\mathbb Z_n$, the tensor $A$ defined below coincides, up to a permutation of its tensor factors, with the structure tensor of the group algebra $\mathbb R[\mathbb Z_n]$, whose multiplication corresponds to cyclic convolution.

\begin{proposition}\label{prop:cyclic-form-norm}
Let $1\leq p\leq q\leq2\leq r\leq\infty$ and put $S:=\frac1p+\frac1q+\frac1r$. For the tensor
$A:=\sum_{i,k\in\mathbb Z_n}e_i\otimes e_{i+k}\otimes e_k$
and its associated trilinear form $T_A:\ell_p^n\times \ell_q^n\times \ell_r^n \to \mathbb R$ with
$T_A(x,y,z):=\sum_{i,k\in\mathbb Z_n}x_i y_{i+k}z_k$,
we have
\begin{equation}\label{eq:exact-cyclic-form-norm}
\|T_A\|=n^{(2-S)_+}.
\end{equation}
Moreover,
\begin{equation}\label{eq:euclidean-norm-A}
\|A\|_{\ell_2^{n^3}}=n.
\end{equation}
\end{proposition}

\begin{proof}
The tensor $A$ in its coefficient array form $A=(a_{ijk})_{i,j,k\in\mathbb Z_n}$ has coefficients
\[
a_{ijk}=
\begin{cases}
1 &:\, j=i+k\pmod n,\\
0 &:\, \text{otherwise}.
\end{cases}
\]

Now for every pair $(i,k)\in\mathbb Z_n^2$ there is precisely one $j\in\mathbb Z_n$ for which $a_{ijk}=1$. Thus $A$ has exactly $n^2$ non-zero coefficients, all of modulus one, and therefore
\[
\|A\|_{\ell_2^{n^3}}
=\left(\sum_{i,j,k\in\mathbb Z_n}|a_{ijk}|^2\right)^{1/2} 
=n.
\]  
This establishes the second assertion of the proposition.

It remains to compute the norm of $T_A$. The key observation is that the trilinear form is associated with cyclic convolution, i.e.,
\[
\begin{aligned}
T_A(x,y,z)
&=\sum_{i,k\in\mathbb Z_n}x_i y_{i+k}z_k\\
&=\sum_{j\in\mathbb Z_n}y_j
  \sum_{i\in\mathbb Z_n}x_i z_{j-i}\\
&\stackrel{\eqref{eq:definition-cyclic-convolution}}{=}\sum_{j\in\mathbb Z_n}y_j(x*z)_j\\
&=\langle x*z,y\rangle.
\end{aligned}
\]
Here, we have used the substitution \(j=i+k\) in \(\mathbb Z_n\) and $\langle\cdot,\cdot\rangle$ denotes the standard Euclidean inner product on $\mathbb R^n$.
We first prove the upper bound. Equip $\mathbb Z_n$ with the normalized counting measure
\[
\mu_n:=\frac1n\sum_{j\in\mathbb Z_n}\delta_j
\] 
and write $\|\cdot\|_{L_s}$ for the corresponding $L_s(\mu_n)$-norm. Recall that the relationship between unnormalized ($\ell_s$) and normalized ($L_s$) norms is
\begin{equation}\label{eq:normalized-counting-norm}
\|x\|_{L_s}=n^{-1/s}\|x\|_s.
\end{equation}
We also use the normalized convolution
\[
(f*_{\mu_n}g)(k):=\frac1n\sum_{i\in\mathbb Z_n}f_i g_{k-i}.
\]
In this setting, Young's inequality \eqref{eq:young-convolution} takes the form
\begin{equation}\label{eq:normalized-young}
\|f*_{\mu_n}g\|_{L_c}\leq\|f\|_{L_a}\|g\|_{L_b}
\end{equation}
whenever
\[
\frac1a+\frac1b=1+\frac1c.
\]
Suppose first that $S\geq2$. Then
\[
\frac1p+\frac1q\geq2-\frac1r=1+\frac1{r^*},
\]
where $r^*$ denotes the Hölder conjugate of $r$ defined by the relation $\frac1r + \frac1{r^*}=1$.
Hence, there exist $a\geq p$ and $b\geq q$ such that
\begin{equation}\label{eq:young-exponents-supercritical}
\frac1a+\frac1b=2-\frac1r=1+\frac1{r^*}.
\end{equation}
Since $a\geq p$ and $b\geq q$, monotonicity of the unnormalized $\ell_s$-norms gives $\|x\|_a\leq\|x\|_p$ and $\|y\|_b\leq\|y\|_q$.
Now define the reflected vector $\widetilde{x}$ by
$\widetilde{x}_i:=x_{-i}$, $i\in\mathbb Z_n$, and put
\[
w_k:=\sum_{i\in\mathbb Z_n}x_i y_{i+k},
\qquad k\in\mathbb Z_n.
\]
Then, by the definition of cyclic convolution \eqref{eq:definition-cyclic-convolution},
\[
(\widetilde{x}*y)_k
=\sum_{j\in\mathbb Z_n}\widetilde{x}_j y_{k-j}
=\sum_{j\in\mathbb Z_n}x_{-j}y_{k-j}
=\sum_{i\in\mathbb Z_n}x_i y_{k+i}
=w_k,
\]
where we used the substitution $i=-j$ in $\mathbb Z_n$.
Thus $w=\widetilde{x}*y$. Moreover, we have
$\|\widetilde{x}\|_a=\|x\|_a$.
Therefore, by Young's inequality \eqref{eq:young-convolution} and \eqref{eq:young-exponents-supercritical},
\[
\|w\|_{r^*}
=\|\widetilde{x}*y\|_{r^*}
\leq\|\widetilde{x}\|_a\|y\|_b
=\|x\|_a\|y\|_b
\leq\|x\|_p\|y\|_q.
\]
Consequently, since
$T_A(x,y,z)=\sum_{k\in\mathbb Z_n}w_kz_k$,
H\"older's inequality gives
\[
\begin{aligned}
|T_A(x,y,z)|
&=\left|\sum_{k\in\mathbb Z_n}w_kz_k\right|\\
&\leq\|w\|_{r^*}\|z\|_r\\
&\leq\|x\|_p\|y\|_q\|z\|_r
\end{aligned}
\]
and so $\|T_A\|\leq1$. Taking $x=y=z=e_0$ shows that $\|T_A\|\geq1$. Thus, whenever $S\geq2$,
\[
\|T_A\|=1.
\]

Suppose now that $S<2$. Then
\[
\frac1p+\frac1q<2-\frac1r.
\]
Since $1\leq p\leq q\leq2$, we have $1/p+1/q\geq1$, while $2-1/r\leq2$. Hence, we may choose $a\leq p$ and $b\leq q$ such that
\begin{equation}\label{eq:young-exponents-subcritical}
\frac1a+\frac1b=2-\frac1r=1+\frac1{r^*}.
\end{equation}
Since $(\mathbb Z_n,\mu_n)$ is a probability space, the monotonicity of $L_s$-norms for $a\leq p$ and $b\leq q$ yields
\[
\|x\|_{L_a}\leq\|x\|_{L_p}
\qquad\text{and}\qquad
\|y\|_{L_b}\leq\|y\|_{L_q}.
\]
Applying Young's convolution inequality in the form \eqref{eq:normalized-young} and then using \eqref{eq:normalized-counting-norm}, we obtain
\begin{align*}
n^{-1-1/r^*}\|w\|_{r^*}
&=\|\widetilde{x}*_{\mu_n}y\|_{L_{r^*}}\\
&\leq\|x\|_{L_a}\|y\|_{L_b}\\
&\leq\|x\|_{L_p}\|y\|_{L_q}\\
&=n^{-1/p-1/q}\|x\|_p\|y\|_q.
\end{align*}
Since $1+1/r^*=2-1/r$, this yields
\begin{equation}\label{eq:cyclic-convolution-subcritical-upper}
\|w\|_{r^*}
\leq n^{2-S}\|x\|_p\|y\|_q.
\end{equation}
H\"older's inequality with conjugate exponents $r^*$ and $r$,
together with \eqref{eq:cyclic-convolution-subcritical-upper},
therefore yields
\[
\begin{aligned}
|T_A(x,y,z)|
&=\left|\sum_{k\in\mathbb Z_n}w_kz_k\right|\\
&\leq\|w\|_{r^*}\|z\|_r\\
&\leq n^{2-S}\|x\|_p\|y\|_q\|z\|_r
\end{aligned}
\]
and hence
\[
\|T_A\|\leq n^{2-S}.
\]

For the reverse inequality, let
\[
x=n^{-1/p}(1,\ldots,1),\qquad
y=n^{-1/q}(1,\ldots,1),\qquad
z=n^{-1/r}(1,\ldots,1).
\]
Then $\|x\|_p=\|y\|_q=\|z\|_r=1$, with the usual interpretation when $r=\infty$, and
\[
T_A(x,y,z)
=\sum_{i,k\in\mathbb Z_n}
 n^{-1/p}n^{-1/q}n^{-1/r}
=n^{2-S}.
\]
Thus $\|T_A\|\geq n^{2-S}$. Together with the upper estimate this proves \eqref{eq:exact-cyclic-form-norm}.
\end{proof}

\begin{cor}\label{cor:new-lower-bound}
Let $1\leq p\leq q\leq2\leq r\leq\infty$. Then
\begin{equation}\label{eq:new-lower-bound}
\left\|\operatorname{id}:\ell_2^{n^3}\longrightarrow X_{p,q,r}^n\right\|
\gtrsim_{p,q,r}
n^{\gamma(p,q,r)},
\end{equation}
where
\begin{equation}\label{eq:new-lower-exponent}
\gamma(p,q,r)
:=\max\left\{
\frac12+\frac1r,
\frac1q,
\min\left\{1,\frac1p+\frac1q+\frac1r-1\right\}
\right\}.
\end{equation}
The third term in \eqref{eq:new-lower-exponent} is furnished by the cyclic convolution tensor. More precisely,
\begin{equation}\label{eq:cyclic-identity-lower}
\left\|\operatorname{id}:\ell_2^{n^3}\longrightarrow X_{p,q,r}^n\right\|
\geq
n^{\min\{1,\,1/p+1/q+1/r-1\}}.
\end{equation}
\end{cor}

\begin{proof}
By duality, Proposition~\ref{prop:cyclic-form-norm} gives
\begin{align*}
\left\|\operatorname{id}:\ell_2^{n^3}\longrightarrow X_{p,q,r}^n\right\|
&=
\left\|\operatorname{id}:(X_{p,q,r}^n)^*\longrightarrow
\ell_2^{n^3}\right\|\\
&=
\sup_{0\neq B\in(X_{p,q,r}^n)^*}
\frac{\|B\|_{\ell_2^{n^3}}}
{\|B\|_{(X_{p,q,r}^n)^*}}.
\end{align*}
Since the dual space $(X_{p,q,r}^n)^*$ is canonically
identified with the space of bounded trilinear forms on
$\ell_p^n\times\ell_q^n\times\ell_r^n$ (see, e.g., \cite[Section~2.2]{Ryan2002} in the two-fold case), the coefficient
array $A$ represents an element of $(X_{p,q,r}^n)^*$
with norm $\|T_A\|$. Hence, evaluating the supremum
at $B=A$ yields
\begin{align*}
\left\|\operatorname{id}:\ell_2^{n^3}
\longrightarrow X_{p,q,r}^n\right\|
&\geq
\frac{\|A\|_{\ell_2^{n^3}}}{\|T_A\|}\\
&=n^{1-(2-S)_+}\\
&=n^{\min\{1,S-1\}},
\end{align*}
where $S=1/p+1/q+1/r$. Note that $S\geq1$ since
$p,q\leq2$, so the exponent is non-negative.
Combining this estimate with the first and third lower
bounds in Lemma~5.6 of \cite{GPSW2017} (which provide
bounds of $n^{1/2+1/r}$ and $n^{1/q}$, respectively)
proves \eqref{eq:new-lower-bound}.
\end{proof}

\begin{rem}\label{rem:comparison-lemma-56}
The cyclic term in \eqref{eq:new-lower-exponent} always dominates
the second term in Lemma~5.6 of \cite{GPSW2017}.
Indeed, writing $S=1/p+1/q+1/r$, if $1/q+1/r\leq1$,
then $S\leq2$ and
\[
\min\{1,S-1\}
=\frac1p+\frac1q+\frac1r-1,
\]
which coincides with the second exponent in Lemma~5.6.
If instead $1/q+1/r>1$, then
\[
\min\{1,S-1\}\geq\frac1p,
\]
whereas the second exponent in Lemma~5.6 equals $1/p$.

To determine when the full lower bound is strictly improved,
observe that for $1/q+1/r\leq1$, the old and new lower
exponents coincide. If $1/q+1/r>1$, then, since $p\leq q$,
the lower exponent from Lemma~5.6 is
\[
\max\left\{\frac12+\frac1r,\frac1p\right\},
\]
whereas \eqref{eq:new-lower-exponent} gives
\[
\max\left\{
\frac12+\frac1r,\min\{1,S-1\}
\right\}.
\]
Consequently, the new lower exponent is strictly larger
if and only if
\[
\min\{1,S-1\}>
\max\left\{\frac12+\frac1r,\frac1p\right\}.
\]
This is equivalent to
\[
p>1,\qquad r>2,\qquad
\frac1q+\frac1r>1,\qquad
\frac1p+\frac1q>\frac32.
\]
Thus our cyclic construction provides a genuine improvement
of the lower bound in Lemma~5.6 of \cite{GPSW2017}
precisely in this non-empty region.
\end{rem}

\begin{example}\label{ex:new-lower-bound}
Take
\[
p=\frac65,\qquad q=\frac54,\qquad r=3.
\]
Then
\[
S=\frac56+\frac45+\frac13=\frac{59}{30}<2,
\]
and Corollary~\ref{cor:new-lower-bound} gives
\[
\left\|\operatorname{id}:\ell_2^{n^3}\longrightarrow
\ell_{6/5}^n\otimes_\pi\ell_{5/4}^n\otimes_\pi\ell_3^n\right\|
\gtrsim n^{29/30}.
\]
By contrast, the three exponents in Lemma~5.6 of \cite{GPSW2017} are
\[
\frac12+\frac13=\frac56,
\qquad
1+\frac56-1=\frac56,
\qquad
\frac45,
\]
so the previously known lower bound has exponent $5/6$. Thus the cyclic convolution tensor improves the exponent from $5/6$ to $29/30$.
\end{example}

\begin{proposition}\label{prop:cyclic-convolution}
Let $1\leq p\leq q\leq2\leq r\leq\infty$ and assume that $\frac1p+\frac1q+\frac1r\geq2$. Then
\begin{equation}\label{eq:identity-order-n}
\left\|\operatorname{id}:\ell_2^{n^3}\longrightarrow X_{p,q,r}^n\right\|
\asymp_{p,q,r} n.
\end{equation}
In fact, the lower estimate in \eqref{eq:identity-order-n} holds with constant one.
\end{proposition}

\begin{proof}
If $S=1/p+1/q+1/r\geq2$, Proposition~\ref{prop:cyclic-form-norm} (more precisely, \eqref{eq:exact-cyclic-form-norm}) gives $\|T_A\|=1$, while $\|A\|_{\ell_2^{n^3}}=n$. Hence, by duality,
\[
\left\|\operatorname{id}:\ell_2^{n^3}\longrightarrow X_{p,q,r}^n\right\|
\geq n.
\]
For the reverse estimate, Lemma~5.5 in \cite{GPSW2017} gives
\begin{equation}\label{eq:known-upper-bound}
\left\|\operatorname{id}:\ell_2^{n^3}\longrightarrow X_{p,q,r}^n\right\|
\lesssim_{p,q,r}n^{\min\{1/q+1/r,1\}}.
\end{equation}
Since $S\geq2$ and $1/p\leq1$, we have
\[
\frac1q+\frac1r\geq2-\frac1p\geq1.
\]
Thus the right-hand side of \eqref{eq:known-upper-bound} is of order $n$, which proves the assertion.
\end{proof}

\begin{cor}\label{cor:cyclic-volume-ratio}
Let $1\leq p\leq q\leq2\leq r\leq\infty$ and suppose that $\frac1p+\frac1q+\frac1r\geq2$.
Then
\begin{equation}\label{eq:volume-ratio-order}
\operatorname{vr}(X_{p,q,r}^n)
\asymp_{p,q,r}n^{1/2-1/r}.
\end{equation}
\end{cor}

\begin{proof}
Theorem~B in \cite{GPSW2017} gives, in the regime $1\leq p\leq q\leq2\leq r\leq\infty$,
\begin{equation}\label{eq:volume-radius}
\operatorname{vol}_{n^3}(B_{X_{p,q,r}^n})^{1/n^3}
\asymp_{p,q,r}n^{-2-1/r}.
\end{equation}
Moreover, $X_{p,q,r}^n$ has enough symmetries. Formula~(2.3) in \cite{GPSW2017} therefore yields
\begin{equation}\label{eq:volume-ratio-formula}
\operatorname{vr}(X_{p,q,r}^n)
\asymp
n^{3/2}
\operatorname{vol}_{n^3}(B_{X_{p,q,r}^n})^{1/n^3}
\left\|\operatorname{id}:\ell_2^{n^3}\longrightarrow X_{p,q,r}^n\right\|.
\end{equation}
Combining \eqref{eq:volume-radius}, \eqref{eq:volume-ratio-formula} and Proposition~\ref{prop:cyclic-convolution}, we obtain
\[
\operatorname{vr}(X_{p,q,r}^n)
\asymp_{p,q,r}
n^{3/2}n^{-2-1/r}n
=n^{1/2-1/r},
\]
which proves \eqref{eq:volume-ratio-order}.
\end{proof}

\begin{rem}
Corollary~\ref{cor:new-lower-bound} also quantifies the improvement outside the exact region of Proposition~\ref{prop:cyclic-convolution}; see Remark~\ref{rem:comparison-lemma-56} and Example~\ref{ex:new-lower-bound}.
\end{rem}

\begin{example}\label{ex:interior-point}
Take
\[
p=\frac{20}{19},\qquad q=\frac54,\qquad r=4.
\]
Then $1<p<q<2<r$ and
\[
\frac1p+\frac1q+\frac1r
=\frac{19}{20}+\frac45+\frac14
=2.
\]
Hence, Proposition~\ref{prop:cyclic-convolution} gives
\[
\left\|\operatorname{id}:\ell_2^{n^3}\longrightarrow
\ell_{20/19}^n\otimes_\pi\ell_{5/4}^n\otimes_\pi\ell_4^n\right\|
\asymp n.
\]
The lower bound from Lemma~5.6 in \cite{GPSW2017} has exponent
\[
\max\left\{\frac34,\frac{19}{20},\frac45\right\}
=\frac{19}{20}.
\]
Thus, the cyclic convolution tensor improves the previously known lower-bound exponent from $19/20$ to $1$ for this choice of parameters. Moreover, Corollary~\ref{cor:cyclic-volume-ratio} gives
\[ 
\operatorname{vr}\left(
\ell_{20/19}^n\otimes_\pi\ell_{5/4}^n\otimes_\pi\ell_4^n
\right)
\asymp n^{1/4}.
\]
\end{example}

\section{A Lower Bound for the Identity Operator via Randomly Signed Cyclic Tensors}\label{sec:random}

We now complement the deterministic cyclic construction by a randomized signing of the same support. The resulting estimate determines the optimal polynomial order of the norm of the identity throughout the full range $1\leq p\leq q\leq2\leq r\leq\infty$, up to a power of a logarithm. In the region covered by Proposition~\ref{prop:cyclic-convolution}, the deterministic result of the previous section is stronger and removes the logarithmic loss completely. For independent Rademacher random variables $\{\varepsilon_{ik}:i,k\in\mathbb Z_n\}$, consider the random tensor
\begin{equation}\label{eq:signed-cyclic-tensor}
A_\varepsilon
:=\sum_{i,k\in\mathbb Z_n}
\varepsilon_{ik}\,e_i\otimes e_{i+k}\otimes e_k
\end{equation}
and the associated random trilinear form
\begin{equation}\label{eq:signed-cyclic-form}
T_\varepsilon(x,y,z)
:=\sum_{i,k\in\mathbb Z_n}
\varepsilon_{ik}x_i y_{i+k}z_k.
\end{equation}
As before,
\begin{equation}\label{eq:signed-euclidean-norm}
\|A_\varepsilon\|_{\ell_2^{n^3}}=n
\end{equation}
for every realization of the signs. We first record the Hilbertian estimate that drives the argument.

\begin{lemma}\label{lem:random-hilbert-bound}
There is a constant $C\in(0,\infty)$ such that
\begin{equation}\label{eq:expected-hilbert-bound}
\mathbb E_\varepsilon
\|T_\varepsilon\|_{\ell_2^n\times\ell_2^n\times\ell_2^n}
\leq C\bigl(\log(en)\bigr)^2.
\end{equation}
Consequently, for every $n\in\mathbb N$ there exists a deterministic choice of signs
$\varepsilon=(\varepsilon_{ik})$ for which
\begin{equation}\label{eq:deterministic-hilbert-bound}
\|T_\varepsilon\|_{\ell_2^n\times\ell_2^n\times\ell_2^n}
\leq C\bigl(\log(en)\bigr)^2.
\end{equation}
\end{lemma}

\begin{proof}
We apply Corollary~1.4 of \cite{Boedihardjo2024}. Define the random three-tensor $X=(X_{ijk})$ by
\[
X_{ijk}:=
\begin{cases}
\varepsilon_{ik}&:\, j=i+k\pmod n,\\
0&:\,\text{otherwise}.
\end{cases}
\]
Its entries are independent, centered, and bounded in absolute value by one; deterministic zero entries cause no difficulty. Moreover, every one-dimensional coordinate fiber of the support contains exactly one nonzero entry. Indeed, fixing any two of the indices $i,j,k$, the relation $j=i+k\pmod n$ uniquely determines the remaining index. Consequently,
\[
\begin{aligned}
\max_{j,k\in\mathbb Z_n}
\left(\sum_{i\in\mathbb Z_n}\mathbb E X_{ijk}^2\right)^{1/2}
&=1,\\
\max_{i,k\in\mathbb Z_n}
\left(\sum_{j\in\mathbb Z_n}\mathbb E X_{ijk}^2\right)^{1/2}
&=1,\\
\max_{i,j\in\mathbb Z_n}
\left(\sum_{k\in\mathbb Z_n}\mathbb E X_{ijk}^2\right)^{1/2}
&=1.
\end{aligned}
\]
Corollary~1.4 of \cite{Boedihardjo2024}, specialized to tensors of order three and side length $n$, therefore yields
\[
\mathbb E_\varepsilon\|X\|_{\mathrm{inj}}
\leq 12\sqrt3+C_0\bigl(\log(en)\bigr)^2
\lesssim \bigl(\log(en)\bigr)^2.
\]
Since the injective norm of $X$ equals the trilinear form norm of $T_\varepsilon$ on $\ell_2^n\times\ell_2^n\times\ell_2^n$, this proves \eqref{eq:expected-hilbert-bound}. The deterministic assertion follows by choosing a realization whose norm does not exceed the expectation.
\end{proof}

Fix from now on a choice of signs satisfying \eqref{eq:deterministic-hilbert-bound}, and put $L_n:=C\bigl(\log(en)\bigr)^2$. The same signed tensor has two further endpoint estimates. First, the one-point-fiber property of the cyclic support gives
\begin{equation}\label{eq:11infty-endpoint}
\|T_\varepsilon\|_{\ell_1^n\times\ell_1^n\times\ell_\infty^n}=1.
\end{equation}
Indeed, for a trilinear form with coefficients $(a_{ijk})$, we have
\[
\|T\|_{\ell_1^n\times\ell_1^n\times\ell_\infty^n}
=\max_{i,j}\sum_k|a_{ijk}|,
\]
and for the tensor in \eqref{eq:signed-cyclic-tensor} the sum on the right-hand side equals one. Second, the Cauchy--Schwarz inequality in the $k$-variable and \eqref{eq:deterministic-hilbert-bound} give
\begin{equation}\label{eq:22infty-endpoint}
\|T_\varepsilon\|_{\ell_2^n\times\ell_2^n\times\ell_\infty^n}
\leq \sqrt n\,L_n.
\end{equation}
Indeed, if $\|x\|_2,\|y\|_2\leq1$, then
\begin{align*}
\sup_{\|z\|_\infty\leq1}|T_\varepsilon(x,y,z)|
&=\sum_k\left|\sum_i\varepsilon_{ik}x_i y_{i+k}\right|\\
&\leq\sqrt n
\left(\sum_k\left|\sum_i\varepsilon_{ik}x_i y_{i+k}\right|^2\right)^{1/2}\\
&\leq\sqrt n\,L_n.
\end{align*}

\begin{thm}\label{thm:universal-log-bound}
Let $1\leq p\leq q\leq2\leq r\leq\infty$. Then
\begin{equation}\label{eq:universal-log-lower}
\left\|\operatorname{id}:\ell_2^{n^3}\longrightarrow X_{p,q,r}^n\right\|
\gtrsim
\frac{n^{\min\{1/q+1/r,\,1\}}}
{\bigl(\log(en)\bigr)^{4(1-1/q)}}.
\end{equation}
Consequently, together with the upper estimate from \cite[Lemma~5.5, (5.8b)]{GPSW2017},
\begin{equation}\label{eq:universal-two-sided}
\frac{n^{\min\{1/q+1/r,\,1\}}}
{\bigl(\log(en)\bigr)^{4(1-1/q)}}
\lesssim
\left\|\operatorname{id}:\ell_2^{n^3}\longrightarrow X_{p,q,r}^n\right\|
\lesssim_{p,q,r}
n^{\min\{1/q+1/r,\,1\}}.
\end{equation}
In particular, the optimal polynomial growth exponent of the norm of the identity is determined throughout this parameter range up to logarithmic factors.
\end{thm}

\begin{proof}
We first consider the region
\begin{equation}\label{eq:nonsaturated-region}
\frac1q+\frac1r\leq1.
\end{equation}
Set
\[
a:=\frac2q-1,\qquad 
b:=\frac2r,\qquad
c:=2-\frac2q-\frac2r.
\]
Under \eqref{eq:nonsaturated-region}, we have $a,b,c\geq0$ and $a+b+c=1$. Moreover,
\[
\frac1q=a+\frac{b+c}{2},
\qquad
\frac1r=\frac b2.
\]
We now interpolate between the three endpoint spaces $\ell_1^n\times\ell_1^n\times\ell_\infty^n$, $\ell_2^n\times\ell_2^n\times\ell_2^n$,
 and $\ell_2^n\times\ell_2^n\times\ell_\infty^n$.
By the multilinear Riesz--Thorin interpolation inequality
for real-valued forms, obtained from the complex
multilinear interpolation theorem
(see, e.g., \cite[Section~4.4]{BerghLofstrom1976})
with an absolute multiplicative constant at most $8$,
and using the endpoint estimates
\eqref{eq:11infty-endpoint},
\eqref{eq:deterministic-hilbert-bound}, and
\eqref{eq:22infty-endpoint}, we obtain
\begin{align}
\|T_\varepsilon\|_{\ell_q^n\times\ell_q^n\times\ell_r^n}
&\leq 8\cdot 1^a L_n^b(\sqrt n\,L_n)^c\nonumber\\
&=8\,n^{1-1/q-1/r}L_n^{2-2/q}\nonumber\\
&\lesssim n^{1-1/q-1/r}
\bigl(\log(en)\bigr)^{4(1-1/q)}.
\label{eq:interpolated-nonsaturated}
\end{align}

Since $p\leq q$, we have $B_{\ell_p^n}\subseteq B_{\ell_q^n}$, and therefore the same upper estimate holds for $\|T_\varepsilon\|_{\ell_p^n\times\ell_q^n\times\ell_r^n}$. By duality and \eqref{eq:signed-euclidean-norm},
\[
\left\|\operatorname{id}:\ell_2^{n^3}\longrightarrow X_{p,q,r}^n\right\|
\geq
\frac{n}{\|T_\varepsilon\|_{\ell_p^n\times\ell_q^n\times\ell_r^n}}
\gtrsim
\frac{n^{1/q+1/r}}
{\bigl(\log(en)\bigr)^{4(1-1/q)}}.
\]
Since $1/q+1/r\leq1$, this establishes \eqref{eq:universal-log-lower} in the present case.

Suppose now that
\begin{equation}\label{eq:saturated-region}
\frac1q+\frac1r\geq1.
\end{equation}
Let $r_0:=q^*$, so that $\frac1q+\frac1{r_0}=1$. The boundary case of \eqref{eq:interpolated-nonsaturated} gives
\[
\|T_\varepsilon\|_{\ell_q^n\times\ell_q^n\times\ell_{r_0}^n}
\lesssim
\bigl(\log(en)\bigr)^{4(1-1/q)}.
\]
Now, condition \eqref{eq:saturated-region} is equivalent to $r\leq r_0$. Hence, we have the inclusion $B_{\ell_r^n}\subseteq B_{\ell_{r_0}^n}$ and consequently
\[
\|T_\varepsilon\|_{\ell_q^n\times\ell_q^n\times\ell_r^n}
\lesssim
\bigl(\log(en)\bigr)^{4(1-1/q)}.
\]
Again using $p\leq q$, duality, and \eqref{eq:signed-euclidean-norm}, we obtain
\[
\left\|\operatorname{id}:\ell_2^{n^3}\longrightarrow X_{p,q,r}^n\right\|
\gtrsim
\frac{n}{\bigl(\log(en)\bigr)^{4(1-1/q)}}.
\]
Since $1/q+1/r\geq1$, this establishes
\eqref{eq:universal-log-lower} in the remaining case. Finally, the upper estimate in \eqref{eq:universal-two-sided} follows from \cite[Lemma~5.5, (5.8b)]{GPSW2017}.
\end{proof}

\begin{rem}\label{rem:relation-two-sections}
Theorem~\ref{thm:universal-log-bound} and the deterministic results of Section~\ref{sec:cyclic} complement each other. The former determines the optimal polynomial growth exponent of the norm of the identity throughout the range $1\leq p\leq q\leq2\leq r\leq\infty$, up to logarithmic factors. In the region $\frac1p+\frac1q+\frac1r\geq2$, Proposition~\ref{prop:cyclic-convolution} yields the stronger, log-free estimate
\[
\left\|\operatorname{id}:\ell_2^{n^3}
\longrightarrow X_{p,q,r}^n\right\|
\asymp_{p,q,r} n.
\]
Thus, the deterministic cyclic construction gives the exact asymptotic order in this region, while the randomized construction determines the optimal polynomial growth exponent throughout the full parameter range.
\end{rem}

\begin{cor}\label{cor:universal-log-volume-ratio}
Let $1\leq p\leq q\leq2\leq r\leq\infty$. Then
\begin{equation}\label{eq:universal-log-vr}
\frac{n^{\min\{1/q-1/2,\,1/2-1/r\}}}
{\bigl(\log(en)\bigr)^{4(1-1/q)}}
\lesssim
\operatorname{vr}(X_{p,q,r}^n)
\lesssim_{p,q,r}
n^{\min\{1/q-1/2,\,1/2-1/r\}}.
\end{equation}
\end{cor}

\begin{proof}
By \eqref{eq:volume-radius} and \eqref{eq:volume-ratio-formula},
\[
\operatorname{vr}(X_{p,q,r}^n)
\asymp_{p,q,r}
 n^{-1/2-1/r}
\left\|\operatorname{id}:\ell_2^{n^3}\longrightarrow X_{p,q,r}^n\right\|.
\]
If $1/q+1/r\leq1$, then
\[
-\frac12-\frac1r+\frac1q+\frac1r
=\frac1q-\frac12,
\]
whereas if $1/q+1/r\geq1$, then
\[
-\frac12-\frac1r+1=\frac12-\frac1r.
\]
These two exponents are respectively the minimum of $1/q-1/2$ and $1/2-1/r$ in the corresponding regions. The assertion now follows from Theorem~\ref{thm:universal-log-bound}.
\end{proof}

The logarithmic loss in Theorem~\ref{thm:universal-log-bound}
arises from the estimate for the injective norm of the
randomly signed cyclic tensor. It is natural to ask whether
this loss is intrinsic to the embedding norm. We conjecture the following. 

\begin{conjecture}\label{conj:log-free}
For every $1\leq p\leq q\leq2\leq r\leq\infty$,
\[
\left\|\operatorname{id}:\ell_2^{n^3}
\longrightarrow X_{p,q,r}^n\right\|
\asymp_{p,q,r}
n^{\min\{1/q+1/r,\,1\}}.
\]
\end{conjecture}

In particular, the conjecture would imply that
\[
\operatorname{vr}(X_{p,q,r}^n)
\asymp_{p,q,r}
n^{\min\{1/q-1/2,\,1/2-1/r\}}
\]
throughout the full parameter range, without logarithmic
corrections.

\section*{Acknowledgement}
JP has been supported by the German Research Foundation under DFG project 516672205 \emph{Limit theorems for the volume of
random projections of $\ell_p$-balls}.
 
ChatGPT (GPT-6, OpenAI) was used for grammar and spelling checks and to assist in the construction of Examples~\ref{ex:new-lower-bound} and~\ref{ex:interior-point}.

\bibliographystyle{abbrv}

\bibliography{CyclicConvolution}

\end{document}